\documentclass{amsart}
\usepackage[utf8x]{inputenc}
\usepackage{amsmath, amsthm, amssymb}
\usepackage[usenames,dvipsnames,svgnames,table]{xcolor}
\usepackage[margin=1.3in]{geometry}
\usepackage{enumitem}
\usepackage{xcolor}
\usepackage{cleveref}
\usepackage[normalem]{ulem}

\newtheorem{theorem}{Theorem}[section]
\newtheorem{proposition}[theorem]{Proposition}
\newtheorem{lemma}[theorem]{Lemma}
\newtheorem{definition}[theorem]{Definition}

\newtheorem{corollary}[theorem]{Corollary}

\newcommand{\R}{\mathbb R}
\newcommand{\T}{\mathbb T}

\newcommand{\eps}{\varepsilon}

\newcommand{\dd}{\, \mathrm{d}}

\newcommand{\vv}{\langle v\rangle}

\newcommand{\supp}{\mbox{supp} \ }

\numberwithin{equation}{section}
\title{Self-generating lower bounds and continuation for the %non-cutoff
 Boltzmann equation}
\author{Christopher Henderson}
\address{Department of Mathematics, University of Arizona, Tucson, AZ 85721}
\email{ckhenderson@math.arizona.edu}
\author{Stanley Snelson}
\address{Department of Mathematical Sciences, Florida Institute of Technology, Melbourne, FL 32901}
\email{ssnelson@fit.edu}
\author{Andrei Tarfulea}
\address{Department of Mathematics, Louisiana State University, Baton Rouge, LA 70803}
\email{tarfulea@lsu.edu}
\thanks{CH was partially supported by NSF grant DMS-2003110. SS was partially supported by a Ralph E. Powe Award from ORAU. AT was partially supported by NSF grant DMS-1816643}

\begin{document}

\maketitle

\begin{abstract}
%	{\color{blue} I'm not sure if I'm happy with this title. Any other suggestions?}
%	{\color{OliveGreen} How about ``Dynamic mass spreading and (coercive) lower bounds for the (non-cutoff) Boltzmann equation''?} {\CH ``Self-generating lower bounds and continuation for the Boltzmann equation''}
%	
	For the spatially inhomogeneous, non-cutoff Boltzmann equation posed in the whole space $\mathbb R^3_x$, we establish pointwise lower bounds that appear instantaneously even if the initial data contains vacuum regions. Our lower bounds depend only on the initial data and upper bounds for the mass and energy densities of the solution. As an application, we improve the weakest known continuation criterion for large-data solutions, by removing the assumptions of mass bounded below and entropy bounded above.
\end{abstract}

\section{Introduction}

The Boltzmann equation is a kinetic equation arising in statistical physics.  Its solution $f(t,x,v)\geq 0$ models the density of particles of a diffuse gas at time $t\in [0,T]$, at location $x\in \R^3$, and with velocity $v\in \R^3$. The equation reads
\begin{equation}\label{e:Boltzmann}
\partial_t f + v\cdot \nabla_x f = Q(f,f) = \int_{\R^3}\int_{{\mathbb S}^2} B(v-v_*,\sigma)
\left( f(v_*')f(v') - f(v_*)f(v) \right) \dd\sigma \dd v_*,
\end{equation}
where $v$ and $v_*$ are post-collisional velocities, and $v'$ and $v_*'$ are pre-collisional velocities, given by
\[
v' = \frac{v+v_*}{2} + \sigma \frac{|v-v_*|}{2}
	\qquad\text{ and }\qquad
v_*' = \frac{v+v_*}{2} - \sigma \frac{|v-v_*|}{2}.
\]
In this article, we focus on the non-cutoff version of \eqref{e:Boltzmann} that includes the physically realistic singularity at grazing collisions. The collision kernel is given by
\[
	B(v-v_*,\sigma) = |v-v_*|^\gamma \theta^{-2-2s} \tilde b(\cos \theta),
		\qquad \text{ where }
	\cos \theta = \sigma \cdot \frac{v-v_*}{|v-v_*|},~
	\gamma > -3,~
	s\in(0,1),
\]
and $\tilde b$ a positive bounded function. 

The main purpose of this article is to prove that pointwise lower bounds for $f$ appear instantaneously, under rather weak assumptions on both the solution and the initial data.   %\sout{Unlike the majority of lower bound results for \eqref{e:Boltzmann},} 
We make no \emph{a priori} assumption of positive mass, except at $t=0$, where uniform positivity in some small ball in $(x,v)$ space is required (but otherwise vacuum regions may exist). The constants in our lower bounds depend only on the initial data and zeroth-order norms of the solution (see \eqref{e:hydro} below).

On physical grounds, gases modeled by \eqref{e:Boltzmann} should be expected to fill vacuum regions instantaneously, so it is important to establish this property under as few assumptions as possible. On a mathematical level, lower bounds for $f$ grant nice coercivity properties to the collision operator $Q(f,f)$, which are a key ingredient of the regularity and existence theory for \eqref{e:Boltzmann} (see, e.g.,~\cite{imbert2020regularity}). Two specific applications we have in mind are:

\begin{enumerate}
\item[(a)] \emph{Continuation.} The recent result of Imbert-Silvestre \cite{imbert2019smooth} (which finished a long program of the two authors and Mouhot, see \cite{silvestre2016boltzmann, imbert2018decay, imbert2018schauder, imbert2019lowerbounds}) showed that solutions to \eqref{e:Boltzmann} can be continued for as long as the mass, energy, and entropy densities of $f$ are bounded above, and the mass density is bounded below, when $\gamma +2s \in [0,2]$. Our main theorem implies that the lower mass bound and upper entropy bound are not needed, reducing the number of required conditions from four to two. See Corollary \ref{c:continuation} below.

\item[(b)] \emph{Local existence.} For the closely related Landau equation, we have shown \cite{HST2019rough} that classical solutions can be constructed for very irregular initial data (bounded and measurable) with pointwise polynomial decay in $v$ of order 5. Pointwise lower bounds for the Landau equation (proven in \cite{HST2018landau}) played an important role in the proof, and we expect Theorem \ref{t:main} to play an analogous role in proving existence of solutions with low regularity initial data for the Boltzmann equation. We plan to explore this question in a future article.

\end{enumerate}

Let us state our results precisely. We work with classical solutions, i.e.~functions $f$ that are $C^1$ in $(t,x)$, $C^2$ in $v$, and satisfy \eqref{e:Boltzmann} in a pointwise sense on $[0,T]\times \R^3 \times \R^3$, with $f(0,x,v) = f_{\rm in}(x,v)$. Our conditional assumptions on $f$ are
\begin{equation}\label{e:hydro}
\begin{split}
&\sup_{t\in [0,T],x\in \R^3} \int_{\R^3} \left(1+ |v|^{\max\{2,\gamma+2s\}}\right) f(t,x,v) \dd v \leq K_0, \qquad \text{ and}\\
&\sup_{t\in [0,T],x\in \R^3} \| f(t,x,\cdot)\|_{L^p(\R^3)} \leq P_0
	\quad \text{ for some } p>\frac{3}{3+\gamma+2s} \quad (\mbox{only if } \gamma + 2s < 0).
%\sup_{t,x} \int_{\R^3} |v|^{\gamma+2s} f(t,x,v) \dd v &\leq G_0 \quad (\mbox{only necessary if } \gamma + 2s > 2),
\end{split}
\end{equation}
For $\gamma + 2s\in [0,2]$, condition \eqref{e:hydro} means the mass and energy densities of $f$ are uniformly bounded.\footnote{In the case $\gamma +2s > 2$, it is known that $f$ enjoys a self-generating bound on $\int_{\R^3} |v|^{\gamma+2s} f \dd v$ in terms of the mass, energy, and entropy densities \cite{cameron2019boltzmann}, but this result requires qualitative assumptions on the solution $f$  (pointwise polynomial decay in $v$ of high order) that we do not make in this article.}  These conditional assumptions grant us some control of the collision operator (see~\Cref{l:Q-s-C2}).  %It will be seen below that $Q(f,f)$ is well-defined for $f\in C^2_v$ satisfying \eqref{e:hydro}. {\CH ????  IS THIS TRUE, THOUGH?  I THINK WE DO NOT ADDRESS THE ISSUE WITH THE SINGULARITY IN CONVOUTION DEFINING $Q_{\rm ns}(f,f)$}

We  assume there are some $\delta, r>0$ and $(x_0,v_0) \in \R^3\times\R^3$ such that the initial data satisfies
\begin{equation}\label{e:mass-core}
	 f_{\rm in} \geq \delta 1_{B_r(x_0)\times B_r(v_0)}.
\end{equation}
Any continuous $f_{\rm in}\not\equiv 0$ satisfies \eqref{e:mass-core}.  A slightly stronger hypothesis, which always follows from~\eqref{e:mass-core} in the spatially periodic case, is
\begin{definition}\label{d:well}
	A function $g: \R^3 \times \R^3 \to [0,\infty)$ is \emph{well-distributed with parameters $R, \delta, r>0$} if, for every $x \in \R^3$, there exists $x_m \in B_{R}(x)$ and $v_m \in B_R(0)$ such that $g \geq \delta 1_{B_r(x_m)\times B_r(v_m)}$.
\end{definition}

Our main result is:
\begin{theorem}\label{t:main}
Let $\gamma\in (-3,1)$ and $s\in (0,1)$.  Suppose that $f:[0,T]\times \R^3\times\R^3 \to \R_+$ is a solution of \eqref{e:Boltzmann} satisfying \eqref{e:hydro} and whose initial data $f(0,\cdot,\cdot) = f_{\rm in}$ satisfies~\eqref{e:mass-core}.  Then% $f$ satisfies lower bounds of the form
\[
	f(t,x,v)
		\geq \mu(t,x) e^{-\eta(t,x)|v|^2}
		\qquad (t,x,v) \in [0,T] \times \R^3 \times \R^3,
\]
where the functions  $\mu(t,x), \eta(t,x)$ are uniformly positive and bounded on any compact subset of $(0,T]\times \R^3$ and %.  The functions $\mu$ and $\eta$ 
depend only on $T$, the constants $K_0$ and $P_0$ in~\eqref{e:hydro}, and $v_0$, $\delta$, and $r$ in~\eqref{e:mass-core}.

If $f_{\rm in}$ is well-distributed with parameters $R$, $\delta$, $r>0$, then $\mu$ and $\eta$ are independent of $x$ and are uniformly positive and bounded on any compact subset of $(0,T]$. % the form $[t_0,T]\times \R^3$ where $t_0>0$.  
In this case, $\mu$ and $\eta$ depend only on $T$, the constants $K_0$ and $P_0$ in~\eqref{e:hydro}, and the constants $R$, $\delta$, and $r$ in~\Cref{d:well}.
%, i.e. for any such compact $K$, there are constants $\mu_K$, $\eta_K > 0$ depending on $K$, $|v_0|$, $\delta$, $r$, $M_0$, $E_0$, $P_0$, and $G_0$, such that $f(t,x,v) \geq \mu_K e^{-\eta_K |v|^2}$ in $K$.

%If the initial data satisfies \eqref{e:mass-core}, then $f$ satisfies lower bounds of the form
%	\[ f(t,x,v) \geq \mu(t,x) e^{-\eta(t,x)|v|^2}, \quad 0\leq t\leq T,\]
%	where the functions  $\mu(t,x), \eta(t,x) >0$ are uniformly positive on any compact subset of $(0,T]\times \R^3$, i.e. for any such compact $K$, there are constants $\mu_K$, $\eta_K > 0$ depending on $K$, $|v_0|$, $\delta$, $r$, $M_0$, $E_0$, $P_0$, and $G_0$, such that $f(t,x,v) \geq \mu_K e^{-\eta_K |v|^2}$ in $K$.
%	
%	If, in addition, $f_{\rm in}$ is well-distributed with parameters $\delta$, $r$, and $R$, then for any $t_0\in (0,T)$, %compact subset $K_{t,x} \subset (0,T]\times \R^3$, 
%	there are constants $\mu_0, \eta_0 >0$ depending on $t_0$, $T$, $\delta$, $r$, $R$, $M_0$, $E_0$, $P_0$, and $G_0$, such that 
%\[ f(t,x,v) \geq \mu_0 e^{-\eta_0 |v|^2}, \quad t\in [t_0,T], x\in \R^3, v\in \R^3.\]
\end{theorem}
Two remarks on the theorem statement are in order. First, the Gaussian asymptotics in $v$ are optimal, as many short-time and close-to-equilibrium solutions are known to satisfy Gaussian upper bounds. Second, it can be seen from the proof that the dependence of our lower bounds on $T$ is mild---in other words, $\mu$ cannot degenerate to $0$ and $\eta$ cannot tend to infinity for any finite $T$, provided the bounds~\eqref{e:hydro} continue to hold. 
%$\mu_0$ and $\nu_0$ cannot degenerate to 0 for any finite $T$, provided the bounds \eqref{e:hydro} continue to hold. 
This is especially important in the following corollary.
 
As an application 
% case $x\in \mathbb T^3$ (i.e.~space periodic periodic boundary conditions) 
we improve the criteria of \cite{imbert2019smooth} for smoothing and continuation in the spatially periodic case by removing the lower bound on the mass and the upper bound on the entropy.

\begin{corollary}\label{c:continuation} 
	Let $f$ be a solution of~\eqref{e:Boltzmann} on $[0,T]\times\T^3\times\R^3$ with initial data $f(0,\cdot,\cdot) = f_{\rm in}$ satisfying~\eqref{e:mass-core}.  Let $\gamma + 2s \in [0,2]$.
\begin{enumerate}
	\item[(a)]  %Let $\gamma + 2s \in [0,2]$. %, and let $f$ be a solution of~\eqref{e:Boltzmann} on $[0,T]\times \mathbb T^3\times \R^3$, 
	Suppose that $f_{\rm in}$ is $C^\infty$ in $(t,x,v)$ and Schwartz class in $v$. If $f$ satisfies the hydrodynamic bounds \eqref{e:hydro}, then, for any $t_0 \in (0,T)$, $q>0$, and any $k$-th order derivative $D^k$ in $(t,x,v)$,%$f$ satisfies regularity estimates of the form
	\[
		\|(1+|v|)^q D^k f\|_{L^\infty([t_0,T]\times \mathbb T^3\times \R^3)}
			\leq C_{q,k,t_0}.
	\]
	The constant $C_{q,k,t_0}$ depends only on $T$, $q$, $k$, $t_0$, $\gamma$, $s$, $K_0$,  and the constants in~\eqref{e:mass-core}.
	\item[(b)] Assume $\gamma \leq 0$ if $s \in (0,1/2)$ and $\gamma < 0$ if $s\in[1/2,1)$. %If $s\in(0,1/2)$, assume $\gamma\in [-2s,0]$. If $s\in [1/2,1)$, assume $\gamma \in ({\CH \{\max\{-3,-(3/2)-2s\}},0)$.
	%Let $f$ be a solution of~\eqref{e:Boltzmann} on $[0,T)\times \mathbb T^3\times \R^3$ with initial data $f(0,\cdot,\cdot)=f_{\rm in}$ satisfying $\vv^\ell \partial_x^\alpha\partial_v^\beta f_{\rm in} \in L^\infty(\mathbb T^3, \R^3)$ for all $\ell\geq 0$ and all multi-indices $\alpha, \beta$. Assume that $f_{\rm in}$ also satisfies \eqref{e:mass-core}. 
	Suppose that $\vv^\ell \partial_x^\alpha\partial_v^\beta f_{\rm in} \in L^\infty(\mathbb T^3, \R^3)$ for all $\ell\geq 0$ and all multi-indices $\alpha, \beta$.  If $f$ cannot be extended to a solution on $[0,T+\eps)$ for any $\eps>0$, then either
	\[
		\lim_{t\to T-} \sup_{x\in \mathbb T^3} \int_{\R^3} f(t,x,v)\dd v = \infty
			\quad \mbox{ or } \quad
		\lim_{t\to T-} \sup_{x\in \mathbb T^3} \int_{\R^3}|v|^2 f(t,x,v)\dd v = \infty.
	\]
\end{enumerate}
\end{corollary}
The extra restrictions on $\gamma$ in part (b) above are inherited from~\cite{morimoto2015polynomial,HST2019boltzmann}; indeed, in its proof, it is necessary to apply a short time existence theorem for initial data with polynomial (rather than exponential or Gaussian) decay in $v$ and, to our knowledge,~\cite{morimoto2015polynomial,HST2019boltzmann} are the only such results.

% come from the need for a local-in-time existence theorem for initial data with polynomial (rather than exponential or Gaussian) decay in $v$ (see \cite{morimoto2015polynomial, HST2019boltzmann}). %The only such results we are available are \cite{morimoto2015polynomial} which requires $s\in (0,1/2)$ and $\gamma \in (-3/2,0]$, and \cite{HST2019boltzmann} which covers $\gamma\in (-3/2-2s,0)$ and any $s\in (0,1)$.

\subsection{Prior lower bounds for the Boltzmann equation}  This topic goes back to Carleman's proof of almost-Gaussian decay in $v$ for a spatially homogeneous hard-spheres model \cite{carleman1933boltzmann}. Still in the homogeneous setting, this was improved to Gaussian decay for cut-off collision kernels by Pulvirenti-Wennberg \cite{pulvirenti1997boltzmann}. Fournier \cite{fournier2001positivity} proved strict positivity for the homogeneous, non-cutoff case, and Mouhot \cite{mouhot2005lowerbounds} derived the first quantitative lower bounds for the inhomogeneous equation (with periodic boundary conditions), obtaining Gaussian decay in the cutoff case, and exponential decay in the non-cutoff case. The more recent work of Imbert-Mouhot-Silvestre \cite{imbert2019lowerbounds} is the first to prove the optimal Gaussian asymptotics for the non-cutoff equation. %, and the current article adapts their strategy (localized in $x$) to spread our lower bounds to large velocities.  {\CH I don't like this phrasing -- it sounds too much like this entire article is just re-doing their thing...}
The current article borrows some techniques from~\cite{imbert2019lowerbounds}.

All results mentioned in the previous paragraph require that the mass density is uniformly bounded from below at every $t$ and $x$ (either as an explicit assumption or by working in the space homogeneous case). % work in a regime where the mass density is bounded from below at every $t$ and $x$ (either by assumption, or in the space homogeneous case, by conservation of mass).
The key feature of our result is that we do not assume any uniform-in-$x$ lower bounds hold, even at $t=0$.

The two works of Briant \cite{briant-arma, briant-krm} are more similar to the present article because they show the instant appearance of exponential lower bounds, despite vacuum regions in the initial data.  Briant's work concerns bounded spatial domains, and his lower bounds depend on the $W^{2,\infty}_v$ norm of $f$; in contrast, the present work applies to solutions in the whole space $\R^3_x$, and our estimates do not depend quantitatively on derivatives of $f$, which is important for the proof of Corollary \ref{c:continuation}. We also note that our proof is substantially simpler than \cite{briant-arma, briant-krm}, although it does not address the setting of mild solutions on domains with boundary as Briant's does.

For more on the history of the Boltzmann equation and its mathematical theory, see \cite{villani2002review, mouhot2018review}.

%Our proofs make essential use of the Carleman decomposition of $Q(f,f)$, which writes $Q$ as the sum of $Q_{\rm s}$, a fractional derivative operator, and $Q_{\rm ns}$, a lower-order convolution term. (See Section \ref{s:prelim}.) This decomposition relies on the cancellation lemma, first discovered in... The recent conditional regularity program of C. Imbert, C. Mouhot, and L. Silvestre (see \cite{silvestre2016boltzmann, imbert2018decay, imbert2018schauder, imbert2019lowerbounds, imbert2019smooth}) uses this decomposition extensively to prove $C^\infty$ regularity of $f$, conditional on uniform control of the mass, energy, and entropy densities from above, and the mass density from below. The lower bound result \cite{imbert2019lowerbounds} referenced above is part of that program. Our Theorem \ref{t:main} allows one to remove the assumptions of entropy bounded above and mass bounded below from all of the results in \cite{silvestre2016boltzmann, imbert2018decay, imbert2018schauder, imbert2019lowerbounds, imbert2019smooth}, as long as the initial data satisfies \eqref{e:mass-core}.

\subsection{Method of proof}\label{s:method}  Briefly, the proof of Theorem \ref{t:main} consists of the following five steps: (1) Propagate the lower bound \eqref{e:mass-core} from $t=0$ to small positive times. (2) Spread the lower bounds from velocities near $v_0$ to all velocities, for small $t$ and $x\approx x_0$. (3) Propagate lower bounds to $(t_1,x_1)$ with $x_1\in\R^3$ arbitrary, and $t_1$ sufficiently small, along trajectories with $x \sim x_0 + t v_1$, using the lower bounds for $f$ near $v_1 = (x_1-x_0)/t_1$. (4) Repeat Step 2 to spread lower bounds to all velocities at $(t_1,x_1)$. (5) Repeat the process to obtain lower bounds up to time $t=T$.

Steps 1 and 3 are achieved by a barrier argument that propagates lower bounds along characteristics of the free transport equation (Lemma \ref{l:push}), and Step 2 is accomplished by adapting the strategy of \cite{imbert2019lowerbounds} to handle lower bounds that are not uniform in $x$ (see Lemma \ref{l:spread}).

\subsection{Comparison with Landau equation}

%\sout{ The Landau equation arises as a limit of the Boltzmann equation as grazing collisions predominate, and is a useful model in plasma physics. The Landau collision operator $Q_L(f,f)$ is a second-order differential operator with nonlocal coefficients, whereas the Boltzmann collision operator is a fractional-order integro-differential operator with a kernel that is itself nonlocal.}

Lower bounds analogous to our Theorem \ref{t:main} were proven for the Landau equation in \cite{HST2018landau}, using a probabilistic method. The argument followed the same steps outlined in Subsection \ref{s:method}, but instead of barriers, the proof proceeded by analyzing the expectation of a stochastic process, which was related to $f$ by a formula of Feynmann-Kac type.  A probabilistic proof should be possible for the Boltzmann equation as well, but one would have to contend with jumps in the corresponding stochastic process  as $Q(f,\cdot)$ is a nonlocal operator for fixed $f$, unlike the Landau collision operator $Q_L(f,\cdot)$, which is local. %\sout{because the diffusion is nonlocal even at the linear level. A more serious (although technical) difficulty would be to ensure the process is well-defined despite the variable-kernel nonlocal diffusion term. In any case,}  
The proof in the present article is simpler, and handles different ranges of $\gamma$ more easily than the probabilistic proof in \cite{HST2018landau}, which required $\gamma < 0$.

Interestingly, the Gaussian asymptotics of the lower bounds in Theorem \ref{t:main} do not hold in general for the Landau equation: in \cite[Proposition 4.4]{HST2018landau}, we showed that when $\gamma < 0$, for certain initial data, $f$ satisfies \emph{upper} bounds proportional to $\exp\{-c|v|^{2+|\gamma|}\}$ for positive times. The proof of Lemma \ref{l:spread} below (or Lemma 3.4 in \cite{imbert2019lowerbounds}) would fail for the Landau equation at the step of bounding $Q(f,f-\varphi)$ from below at a crossing point (where $\varphi$ is a barrier). %{\CH Both ``kinds'' of nonlocality ($Q(\cdot,\cdot)$ is nonlocal in both of its arguments) are used to obtain a good lower bound for this term. For the Landau equation, one can obtain a good sign for $Q_L(f,f-\varphi)$ at the first crossing point via $D_v^2(f-\varphi)\geq 0$, but there is no hope for a positive lower bound since $Q_L(\cdot,\cdot)$ is only nonlocal in its first argument.} 
 The nonlocality of $Q(\cdot,\cdot)$ with respect to both of its arguments is used to obtain a good lower bound for this term. Since $Q_L(\cdot,\cdot)$ is local in its second argument, the analogous lower bound fails for the Landau equation.
 
% For the Landau equation, since $Q_L(f,g)$ is nonlocal with respect to $f$ but not $g$, %(i.e. it is a second-order differential operator with nonlocal coefficients depending on $f$),
%% one can obtain a good sign for $Q_L(f,f-\varphi)$ at the first crossing point via $D_v^2(f-\varphi)\geq 0$, but there is no hope for a positive lower bound.

%%{\color{blue}[NOTE: We could optionally remove the parenthetical part of the last sentence, to be more concise.]}
%{\color{OliveGreen}[I think we should remove it too. We're not here to get into the details of Landau]}. {\CH Removed}  {\CH I think we should cut down on the last sentence as well:  what about ``...a good lower bound for this term; unfortunately, the Landau equation is local in its second argument and a similarly strong lower bound does not hold.''  ???} 

\subsection{Notation} We say a constant is \emph{universal} if it depends only on $\gamma$, $s$, and $K_0$ (if $\gamma + 2s<0$, we additionally allow dependence on $p$ and $P_0$).  We write $C$ to be any positive, universal constant changing line-by-line. Additional dependence is denoted with subscripts, e.g.~$C_r$ depends on universal constants as well as $r$ and may also change line-by-line.

%Often, constants depend on other parameters in a way in which we do not track.  In this case, we display dependence by using the parameter as a subscript (e.g., $C_r$ is a positive constant universal except for additional dependence on $r$ and may change line-by-line). 

We write $\langle a \rangle = \sqrt{1+|a|^2}$ for any vector or scalar $a$.

%, i.e. the constants in \eqref{e:hydro}. {\color{blue}[If we want, we could define ``universal'' right after \eqref{e:hydro}, and use the word universal in the theorem statement.]} We write $A \lesssim B$ to mean $A\leq CB$ for a universal constant $C$, and similarly, $A\approx B$ means $A\lesssim B$ and $B \lesssim A$. We write $\langle a \rangle = \sqrt{1+|a|^2}$ for any vector or scalar $a$.

\subsection{Outline of the paper} In Section \ref{s:prelim}, we review some useful properties of the Carleman decomposition. Section \ref{s:main} contains the proof of Theorem \ref{t:main}, and Section \ref{s:continuation} proves Corollary \ref{c:continuation}.

\section{Preliminaries}\label{s:prelim}

\subsection{Carleman decomposition of $Q(f,g)$.}

By adding and subtracting $g(v)f(v_*')$ inside the integral, we can write
\begin{equation}
Q(f,g) = Q_{\rm s}(f,g) + Q_{\rm ns}(f,g)
\end{equation}
where
\begin{equation}
\begin{split}
	&Q_{\rm s}(f,g) = \int_{\R^3} \int_{{\mathbb S}^2} (g(v')-g(v)) f(v_*') B(|v-v_*|,\sigma)\dd \sigma \dd v_*
		\qquad\text{and}\\
	&Q_{\rm ns}(f,g) = g(v) \int_{\R^3} \int_{{\mathbb S}^2} (f(v_*')-f(v_*)) B(|v-v_*|,\sigma) \dd \sigma \dd v_*.
\end{split}
\end{equation}
Following \cite{silvestre2016boltzmann} (see also \cite{alexandre2000new}), $Q_{\rm s}$ is an integro-differential operator with
 kernel $K_f$:
\begin{lemma}{\cite[Section 4]{silvestre2016boltzmann}}\label{l:Q1}
The operator $Q_{\rm s}(f,g)$ can be written
\begin{equation}\label{e:Q1-new}
Q_{\rm s}(f,g) = \int_{\R^3} (g(v')-g(v)) K_f (v, v') \dd v',
\end{equation}
with kernel  
\begin{equation}\label{e:kernel}
K_f(v,v') = \frac {1} {|v'-v|^{3+2s} }\int_{\{v_*' : (v-v_*')\cdot (v'-v) = 0\}} f(v_*') |v-v_*'|^{\gamma+2s+1} \tilde b(\cos\theta) \dd v_*',
\end{equation}
where $\tilde b$ is uniformly positive and bounded.
%In particular, if $f:\R^3\to \R_+$ is such that, for any $R>0$, there is some $K_0$ with
%\begin{equation}
%\sup_{v \in B_R} \int_{\R^3} |w|^{\gamma+2s} f(v+w) \dd w \leq K_0,
%\end{equation}
%then $K_f(v,v')$ satisfies
%\begin{equation*}
%\int_{B_{2r} \setminus B_r} K_f(v, v+z) \dd z \leq CK_0 r^{-2s},
%\end{equation*}
%for all $r>0$ and $v\in B_R$, where $C>0$ depends only on the collision kernel.
%Furthermore, $K_f(v,v')$ is symmetric about $v$:
%\[ K_f(v,v+z)  = K_f(v,v-z).\]
\end{lemma} 

Technically, the integral $\int_{\R^3} (g(v') - g(v))K_f(v,v')\dd v'$ is understood in a principal value sense when $s\in [\frac 1 2, 1)$, but we work with functions smooth enough that the integral is always well-defined. Therefore, we abuse notation by omitting the ``p.v.''

%For $f$ satisfying \eqref{e:hydro}, Lemma \ref{l:Q1} easily implies
%\begin{equation}\label{e:annulus}
%\int_{B_{2r} \setminus B_r} K_f(v, v+z) \dd z \leq \Lambda\vv^{(\gamma+2s)_+} r^{-2s},
%\end{equation}
%for all $r>0$ and $v\in \R^3$, with $\Lambda > 0$ depending only on $M_0$, $E_0$, and $P_0$.

We also need the following pointwise upper bound on $Q_{\rm s}(f,g)$:
\begin{lemma}\cite[Lemma 2.3]{imbert2019lowerbounds}\label{l:Q-s-C2}
	If $g$ is bounded and $C^2$, then 
	\[ |Q_{\rm s}(f,g)(t,x,v)|
		\leq C\left(\int_{\R^3} |w|^{\gamma+2s} f(t,x,v-w)\dd w\right)
			\|g\|_{L^\infty(\R^3)}^{1-s}
			\|D_v^2 g\|_{L^\infty(\R^3)}^{s}.\]
\end{lemma}
For $f$ satisfying \eqref{e:hydro}, Lemma \ref{l:Q-s-C2} easily implies
\begin{equation}\label{e:Qs-bound}
|Q_{\rm s}(f,g)(t,x,v)|  \leq \Lambda\vv^{(\gamma+2s)_+} \|g\|_{L^\infty(\R^3)}^{1-s} \|D_v^2 g\|_{L^\infty(\R^3)}^{s},
\end{equation}
with $\Lambda > 0$ depending only on $K_0$ and $P_0$.

The second term in this decomposition is a lower-order convolution term, thanks to the Cancellation Lemma (see \cite{alexandre2000new, alexandre2000entropy, villani1999boltzmann}). We quote from \cite{silvestre2016boltzmann} for convenience:
\begin{lemma}{\cite[Lemmas 5.1 and 5.2]{silvestre2016boltzmann}}\label{l:Q2}
There is a constant $C>0$ such that %The term $Q_{\rm ns}(f,g)$ can be written %{\CH Should we cite something older than this?  Some of the hardcore kinetic people might get upset at us attributing the cancellation lemma to Luis...}
\[
	Q_{\rm ns}(f,g) = C g(v) \int_{\R^3} |z|^\gamma f(v-z) \dd z.
\]
%where $C>0$ depends only on $\gamma$, and the collision kernel.
\end{lemma}
Since $Q_{\rm ns}(f,g) \geq 0$ whenever $g\geq 0$, this term has a good sign in barrier arguments used to prove lower bounds.

\section{Proof of the main result}\label{s:main}

\subsection{Propagating lower bounds forward in time}

The following key lemma is used both to preserve a mass core near $(x_0,v_0)$ for short times (which corresponds to $\tau = 1$ in the statement of the lemma), and to push lower bounds to different locations in $x$ via free transport.

\begin{lemma}\label{l:push}
Fix $\tau>0$ and $f$ solving~\eqref{e:Boltzmann} on $[0,T]\times \R^3 \times \R^3$. If $f(0,x,v) \geq \delta 1_{\{|x-x_0|<r, |v-v_0|<r/\tau\}}$ for some $(x_0,v_0)\in\R^6$ and $\delta, r>0$, then the lower bound
\[ f(t,x,v) \geq \frac \delta 2\]
holds whenever $0\leq t\leq \min\{T,\tau\}$ and, for a universal constant $C_{\ref{l:push}}$,
\begin{equation}\label{e:txv}
 \frac {|v-v_0|^2}{r^2/\tau^2} + \frac{|x-x_0 - tv|^2}{r^2} < \frac {1} 4 \quad \mbox{ and } \quad t < \frac{C_{\ref{l:push}} r^{2s}} { \tau^{2s} \langle |v_0| + r/\tau\rangle^{(\gamma+2s)_+} }.
 \end{equation}
\end{lemma}

\begin{proof}
Consider the function
\begin{equation}\label{e:subsolution}
\underline{f}(t,x,v):= -c_1 t + c_2 \psi\left(1 - \frac{|v-v_0|^2}{r^2/\tau^2} - \frac{|x-x_0-tv|^2}{r^2}\right),
\end{equation}
with $c_1, c_2>0$ chosen later. Here $\psi$ is a smooth approximation of the ``positive part'' function; that is, a smooth, increasing function such that
\begin{equation*}
\psi(s) = \begin{cases}
 0, &\text{if  }  s \leq 0,\\
 s, &\text{if  }  s \geq 1/2.
\end{cases}
\end{equation*}
%{\color{OliveGreen}We want to show $\underline f$ is a subsolution to the linear Boltzmann equation in the sense of Lemma \ref{l:max-principle}. It is clear that
%$\underline f < 0$ if $|v-v_0| \geq r/\tau$, from which it also follows that $\underline f < 0$ if $|v-v_0| < r/\tau$ and $|x-x_0| > 2r + \tau|v_0|$.
%So if $K := B_{2r+\tau|v_0|}(x_0) \times B_{r/\tau}(v_0)$ (which is compact), then $\underline f(t,x,v) < 0$ whenever $(x,v) \in K^{\rm c}$.}
We wish to show that $\underline f$ is a subsolution to the linear Boltzmann equation, at least at points where it is positive. Assume that $(t,x,v)$ is such that $\underline f(t,x,v) > 0$. 
We clearly have $Q_{\rm s}(f,\underline f) = Q_{\rm s}(f, \underline f + c_1 t)$, so that \eqref{e:Qs-bound} and the nonnegativity of $Q_{\rm ns}(f,\underline f)(t,x,v)$ (which holds because $\underline f(t,x,v)>0$) imply
\[Q(f,\underline f)(t,x,v) \geq -\Lambda \vv^{(\gamma+2s)_+} \|\underline f+c_1 t\|_{L^\infty_v(\R^3)}^{1-s}\|D_v^2 \underline f\|_{L^\infty_v(\R^3)}^s, \]
with $\Lambda$ universal. Clearly, $\|\underline f + c_1t\|_{L^\infty_v(\R^3)} = c_2$. Next, for any $v\in \R^3$, with the shorthand $h_r = 1 - \tau^2|v-v_0|^2/r^2 - |x-x_0-tv|^2/r^2$, 
\[ 
\begin{split}
	|D_v^2 \underline f(v)| &= |4 c_2 \psi''(h_r)r^{-4}(\tau^2(v-v_0)-t(x-x_0-tv)) \otimes (\tau^2(v-v_0) - t(x-x_0-tv))\\
	&\qquad  + 2c_2 \psi'(h_r) r^{-2}\delta_{ij} (\tau^2 + t^2)|\\
	&\leq Cc_2 \left(|\psi''(h_r)| (t^2+\tau^2) r^{-2} + |\psi'(h_r)|r^{-2}\tau^2\right)
	\leq Cc_2 r^{-2} \tau^2 .
\end{split}
\]
We have used $t\leq \tau$, and that $\psi''(h_r) = 0$ if $\tau^2|v-v_0|^2/r^2 + |x-x_0-tv|^2/r^2 > 1$. Therefore,
	\[
		Q(f, \underline f)(t,x,v)
			\geq -\Lambda\vv^{(\gamma+2s)_+} c_2 \tau^{2s} r^{-2s} .
	\]
Let $c_1 = 2\Lambda \langle |v_0| + r/\tau \rangle^{(\gamma+2s)_+} c_2\tau^{2s} r^{-2s}$.  Thus, %$c_1 > \Lambda\vv^{(\gamma+2s)_+} c_2 \tau^{2s} r^{-2s}$ for all $v\in B_{r/\tau}(v_0)$, and we find%at $f(0,x,v) \geq \underline f(0,x,v)$, 
%and $c_1 = 2\Lambda \langle |v_0| + r/\tau \rangle^{(\gamma+2s)_+} c_2\tau^{2s} r^{-2s}$. 
\begin{equation}\label{e:Q1-subsolution}
\partial_t \underline f + v \cdot \nabla_x \underline f = -c_1 < Q(f,\underline f) \qquad\mbox{ for all } v\in B_{r/\tau}(v_0).%\underline f(t,x,v) > 0,
\end{equation}
%provided $c_1 > \Lambda\vv^{(\gamma+2s)_+} c_2 \tau^{2s} r^{-2s}$ for all $v\in B_{r/\tau}(v_0)$. We take $c_2 = 3\delta/4$ in \eqref{e:subsolution}, %so that $f(0,x,v) \geq \underline f(0,x,v)$, 
%and $c_1 = 2\Lambda \langle |v_0| + r/\tau \rangle^{(\gamma+2s)_+} c_2\tau^{2s} r^{-2s}$. 

Now, we claim $f > \underline f$ for all $(t,x,v)$ such that $\underline f(t,x,v) > 0$. By choosing $c_2 = 3\delta/4$, this claim is true for $t=0$. If the claim fails, then there is a first crossing point $(t_0,x_0,v_0)$ with $\underline f(t_0,x_0,v_0) >0$, such that $f(t_0,x_0,v_0) = \underline f(t_0,x_0,v_0)$ and $f(t,x,v) > \underline f(t,x,v)$ whenever $\underline f(t,x,v) > 0$ and $t< t_0$. The strict positivity of $t_0$ follows from the compact support of $\underline f(t,\cdot,\cdot)$ for each $t$. We also have $f(t_0,x,v) \geq \underline f(t_0,x,v)$ for all $(x,v)\in \R^6$.

Letting $g = f - \underline f$, we have $\partial_t g(t_0,x_0,v_0) \leq 0$ and $\nabla_x g(t_0,x_0,v_0) = 0$, so that \eqref{e:Q1-subsolution} implies
\begin{equation}\label{e.c1}
	0 \geq (\partial_t + v_0\cdot \nabla_x)g(t_0,x_0,v_0) > Q(f,g).
\end{equation}
Next, since $g(t_0,x,v) \geq 0 = g(t_0,x_0,v_0)$ for all $(x,v)\in\R^6$, we have
\[ Q_{\rm s}(f,g)(t_0,x_0,v_0) = \int_{\R^3} K_f(t_0,x_0,v,v')(g(t_0,x_0,v') - g(t_0,x_0,v_0)) \dd v'\geq 0,\]
and  $Q_{\rm ns}(f,g)(t_0,x_0,v_0) \geq 0$ by Lemma \ref{l:Q2}.  Thus, $Q(f,g) \geq 0$, contradicting~\eqref{e.c1}.

This contradiction implies $f \geq \underline f$ whenever $\underline f(t,x,v) > 0$. The claim then follows by choosing $C = \Lambda/24$ and using the definition of $\underline f$.
%In particular, for $(t,x,v)$ satisfying \eqref{e:txv} with $C = \Lambda/24$, our choice of $\psi$ implies $\psi(3/4) = 3/4$, and
%\[f(t,x,v) \geq \underline f(t,x,v)  > -c_1 t +  \frac{3} 4 c_2\geq -\frac 1 {16} \delta + \frac 9 {16} \delta \geq \frac {\delta} 2.\]
\end{proof}

\subsection{Spreading lower bounds to all velocities}\label{s:all-velocities}

%The most common strategy for proving lower bounds that are global in $v$ is: first, establish an initial plateau (i.e. a lower bound that holds for small velocities). This is provided by 

First, we require a lemma that spreads local lower bounds to a larger domain in $v$, at the cost of shrinking the $x$-domain where the lower bounds hold. The proof strategy and notation are similar to \cite[Lemma 3.4]{imbert2019lowerbounds}. The differences are: first, that our initial lower bound is not uniform in $x$, so we need to include a cutoff in $x$ in our barrier function; and second, that our solution $f$ is not yet known to be strictly positive everywhere, so that we must make our barrier strictly negative for large $v$ and $x$ to control the location of the first crossing point.

\begin{lemma}\label{l:spread}
	Suppose that $f$ is a solution of~\eqref{e:Boltzmann}on $[0,T]\times\R^3\times\R^3$ %For any $T_0\in (0,1]$ {\CH Why do we want $T_0 \leq 1$?  Don't we want it $\leq T$?} and any solution $f$ to \eqref{e:Boltzmann}
	satisfying \eqref{e:hydro} as well as
	\[
		f \geq \ell
			\qquad  \text{on }
			[0,T]\times B_\rho(0)\times B_R(0),
	\]
	for some $\ell>0$, $\rho>0$, and $R\geq 1$, there is a universal constant $c$ such that for any $\xi \in (0, 1-2^{-1/2})$ such that $\xi^q R^{3+\gamma} \ell < 1/2$ with $q = 5 + 2(\gamma+2s)$, there holds
	\[
		f
			\geq c \xi^q R^{3+\gamma} \ell^2\min\{t, R^{2s-(\gamma+2s)_+}\xi^{2s}+ R^{-1}\rho\}
			\qquad \text{ on } [0,T]\times B_{\rho/2}(0)\times B_{\sqrt 2 (1 - \xi)R}.
	\]
\end{lemma}
Note that $\xi < 1-2^{-1/2}$ implies $\sqrt 2 (1-\xi) > 1$.
\begin{proof}
	We may assume that $T\leq 1$.  Otherwise, we may shift $f$ in time to establish the claim on time intervals of the form $[T_0 - 1, T_0]$ for $T_0 \in [1, T]$.
	
	First we construct suitable cutoff functions in $v$ and in $x$. Let $\varphi_\xi(v)$ be a smooth function such that $\varphi_\xi = 1$ in $B_{\sqrt 2(1-\xi)}(0)$,  $\varphi_\xi = 0$ outside $B_{\sqrt 2 (1-\xi/2)}$, and $\|D^2 \varphi_\xi\|_{L^\infty} \leq 10 \xi^{-2}$. Define $\varphi_{\xi,R}(v) = \varphi_\xi(v/R)$. Next, let $\psi_\rho$ be a smooth, radially decreasing function with $\psi_\rho(x) = 1$ in $B_{\rho/2}$ and $\psi_\rho(x) = 0$ outside $B_\rho(0)$ satisfying $|\nabla_x \psi_\rho|\leq 4 \rho^{-1}$.
	
	Applying \eqref{e:Qs-bound}, we have, for some constant $C_1>0$,
	\begin{equation}\label{e:Q-phi-bound}
	 |Q_{\rm s}(f, \varphi_{\xi,R})(t,x,v)| \leq C_1\vv^{(\gamma+2s)_+} (\xi R)^{-2s}.
	 \end{equation}

	Our barrier in the $t$ variable is defined by
	\[ \begin{split} \tilde \ell(t) &:= \alpha \xi^q R^{3+\gamma} \ell^2 \left(\frac{1-e^{-Kt}}{K}\right),
	\qquad\text{where }
	K := 4\sqrt{2}R \rho^{-1} + C_1  \langle \sqrt 2 R\rangle^{(\gamma+2s)_+} (R\xi)^{-2s},\end{split}\]
	and where $\alpha\in (0,1)$ is chosen later. Note that $T\leq 1$ and the smallness condition on $\xi$ together imply $\tilde \ell (t) \leq \ell/2$ for all $t\in [0,T_0]$.
	
	For small $\eps>0$ (later we take $\eps \to 0$), our goal is to prove
	\begin{equation}\label{e:f-lower}
	f(t,x,v) > h(t,x,v) := \tilde \ell(t)\varphi_{\xi,R}(v)\psi_\rho(x) - \eps, \quad 0\leq t \leq T_0, (x,v)\in \R^6.
	\end{equation}
	This inequality holds at $t=0$ because $f\geq 0$ and $\tilde \ell(0) = 0$. If \eqref{e:f-lower} is false, then there is a first crossing point $(t_0,x_0,v_0)$ with $t_0 \in (0,T_0]$, such that $f(t_0,x_0,v_0) = h(t_0,x_0,v_0)$, and $f(t,x,v) \geq h(t,x,v)$ for all $t\leq t_0$ and $(x,v)\in \R^6$. Since $f \geq 0$, the crossing point must satisfy $x_0 \in \supp \psi_\rho$ and $v_0\in \supp \varphi_{\xi,R}$, i.e. $|x_0|\leq \rho$ and $|v_0|\leq \sqrt 2 (1-\xi/2)R$. The strict positivity of $t_0$ follows from the compactness of $\supp \psi_\rho \times \supp \varphi_{\xi,R}$ (recall that $f$ is a classical solution and, hence, continuous up to time $t=0$).

%	We claim that the crossing point must satisfy $|v_0|> R$. Indeed, {\CH if $|v_0| \leq R$, then, recalling that $|x_0| \leq \rho$, we have $f(t_0,x_0,v_0) = h(t_0,x_0, v_0) \leq \ell/2 < \ell \leq f(t_0,x_0, v_0)$.  Here, the first inequality follows by construction of $h$ and the last follows by assumption.  This is a contradiction.}
%	
%	 $|v_0|\leq R$, then, {\CH recalling that $|x_0| \leq \rho$, we 
%	
%	
%	 we cannot have $|x_0|\leq \rho$, because $f \geq \ell$ when $(x,v)\in B_{\rho}\times B_R$ by assumption and $\tilde \ell(t){\color{OliveGreen}\varphi}_{\xi,R}(v)\psi_\rho(x) \leq \ell/2$ everywhere.
	
	%{\CH ?????  As we have already shown that $x_0 \in \supp \psi_\rho$, we have $|x_0| < \rho$, $|v_0| \in (R,\sqrt 2 (1 - \xi/2)R)$. ?????}

	Since $(t_0,x_0,v_0)$ is the first crossing point, we have $\tilde \ell'(t_0)\varphi_{\xi,R}(v_0)\psi_\rho(x_0) = \partial_t h(t_0,x_0,v_0) \geq \partial_t f(t_0,x_0,v_0)$ and $\nabla_x f(t_0,x_0,v_0) = \nabla_x h(t_0,x_0,v_0) = \tilde \ell(t_0)\varphi_{\xi,R}(v_0)\nabla_x \psi_\rho(x_0)$. From this,~\eqref{e:Boltzmann} and the fact that $Q_{\rm ns}(f,f) \geq 0$, we conclude that
	\begin{equation}\label{e:supersoln}
	\tilde \ell'(t_0){\color{OliveGreen}\varphi}_{\xi,R}(v_0) \psi_\rho(x_0) + \tilde \ell(t_0)\varphi(v_0)v_0\cdot \nabla_x\psi_\rho(x_0)\geq Q_{\rm s} (f,f)(t_0,x_0,v_0).
	\end{equation}
	Using the linearity of $Q_{\rm s}(f,\cdot)$,~\eqref{e:Q-phi-bound}, and the fact that $Q_{\rm s}(f,g+\eps) = Q_{\rm s}(f,g)$, we find
	\[ \begin{split}
		Q_{\rm s}(f,f) &= Q_{\rm s}(f,f - h) + \tilde \ell(t) \psi_\rho(x)Q_{\rm s}(f, \varphi_{\xi,R}) %+ \eps Q_{\rm s}(f, \varphi_{\xi,R}),\\
		\geq Q_{\rm s}(f,f-h) - \tilde \ell(t)\psi_\rho(x) C_1 \vv^{(\gamma+2s)_+} (\xi R)^{-2s}.
		\end{split}\]
	
	On the other hand, using that $\tilde \ell'(t) = \alpha\xi^q R^{3+\gamma}\ell^2 - K\tilde \ell(t)$, that $\phi_{\xi,R}, \psi_{\rho} \leq 1$, and that $|\nabla_x \psi_\rho| \leq 4\rho^{-1}$ in~\eqref{e:supersoln}, we find
%	
%	
%	Note also that
%	\[ \tilde \ell'(t) = \alpha\xi^q R^{3+\gamma}\ell^2 - K\tilde \ell(t),\]
%	so that \eqref{e:supersoln} becomes, using $\varphi_{\xi,R}\leq 1$ and $\psi_\rho \leq 1$,
	\[\begin{split}
	 \alpha \xi^q R^{3+\gamma} \ell^2 - K \tilde \ell(t_0)
	 	+ \tilde \ell(t_0) C_1 \langle v_0\rangle^{(\gamma+2s)_+} (\xi R)^{-2s}
		+ 4|v_0| \rho^{-1} \tilde \ell(t_0)
		&\geq Q_{\rm s}(f,f-h)(t_0,x_0,v_0).
	 \end{split} \]
	Now, from $|v_0| \leq \sqrt 2 R$ and the definition of $K$, we cancel terms and obtain
	\[\begin{split}
		&\alpha \xi^q R^{3+\gamma} \ell^2 \geq Q_{\rm s}(f, f- h)(t_0,x_0,v_0)\\
	& = \int_{\R^3} \int_{\{v_*':(v_*'-v_0)\cdot(v_0-v') = 0\}} (f(v') - h(t_0,x_0,v')) f(v_*') \frac{|v'-v_*'|^{\gamma+2s+1}}{|v_0-v'|^{3+2s}} \tilde b(\cos\theta)\dd v_*'  \dd v'\\
	&\geq \int_{B_R} \int_{\{v_*':(v_*'-v_0)\cdot(v_0-v') = 0\}} 1_{B_R}(v_*') (f(v') - h(t_0,x_0,v')) f(v_*') \frac{|v'-v_*'|^{\gamma+2s+1}}{|v_0-v'|^{3+2s}} \tilde b(\cos\theta)\dd v_*'  \dd v'.
	\end{split}\]
	In the last inequality, we used the nonnegativity of the integrand to discard the parts of the integral with $v'\not\in B_R$ and $v_*'\not\in B_R$.
	
	For $v'\in B_R$, we have $(f-h)(t_0,x_0,v')\geq f(t_0,x_0,v') - \tilde \ell(t_0)\psi_\rho(x_0)\varphi_{\xi,R}(v')$. Furthermore, since $|x_0|\leq \rho$, on the domain of integration, we have $f(t_0,x_0,v_*') \geq \ell$, and $f(t_0,x_0,v') - h(t_0,x_0,v') \geq \ell - \tilde \ell(t_0) \geq \ell/2$, so our inequality becomes
	\begin{equation}\label{e:c2}
		\alpha \xi^q R^{3+\gamma} \ell^2
			\geq \frac{\ell^2}{C R^{3+2s}}\int_{B_R} \int_{\{v_*':(v_*'-v_0)\cdot(v_0-v') = 0\}} 1_{B_R}(v_*') |v'-v_*'|^{\gamma+2s+1} \tilde b(\cos\theta)\dd v_*'  \dd v'.
	\end{equation}
	At this point, we may quote verbatim the calculations of \cite[Lemma 3.4]{imbert2019lowerbounds} to obtain
	\begin{equation}\label{e:c3}
		%\alpha \xi^q R^{3+\gamma} \ell^2
		\ell^2 R^{-3-2s}\int_{B_R} \int_{\{v_*':(v_*'-v_0)\cdot(v_0-v') = 0\}} 1_{B_R}(v_*') |v'-v_*'|^{\gamma+2s+1} \tilde b(\cos\theta)\dd v_*'  \dd v'	
			\geq \beta \xi^q R^{3+\gamma} \ell^2,
	\end{equation}
	for some $\beta>0$ independent of the free parameter $\alpha\in (0,1)$.  Choosing $\alpha$ sufficiently small,~\eqref{e:c2} and~\eqref{e:c3} yield a contradiction.  We note that our choice of $\alpha$ is independent of $\eps$.
	
	This establishes $f(t,x,v) > \tilde \ell(t) \psi_\rho(x) \varphi_{\xi,R}(v) - \eps$.   Taking $\eps\to 0$, we conclude the proof; that is, for $t\in [0,T_0], x\in B_{\rho/2}(0), v\in B_{\sqrt 2 (1-\xi)R}(0)$,
	\[ f(t,x,v) \geq \tilde \ell(t) \geq \alpha \xi^q R^{3+\gamma}\ell^2\frac{1-e^{-Kt/2}}{K} \geq c \xi^q R^{3+\gamma}\ell^2\min\{t,K^{-1}\}.\]
\end{proof}

\begin{proposition}\label{p:gaussian}
	Let $f$ be a solution to~\eqref{e:Boltzmann} on $[0,t_0]\times \R^6$ satisfying
	\[
		f \geq \ell
			\qquad \text{ on }
			\{0\leq t \leq t_0,  |x-x_0-tv_0|<r, |v-v_0|<r\}.
	\]
	Fix $\underline t \in (0,t_0)$.  Then there are constants $c_1,c_2>0$ depending only on universal constants, $\underline t$, and $r$, such that %  $\gamma$, $s$, $M_0$, $E_0$, $P_0$, and $G_0$, such that
	\[
		f(t,x,v)
			\geq c_1 \ell e^{-c_2|v-v_0|^2}
			\quad \mbox{ if } t\in [\underline t,t_0], v\in \R^3, |x-x_0-tv_0|<r/2.\] 
\end{proposition}

\begin{proof} 
	First, we recenter around the origin by defining
	\[ \tilde f(t,x,v) = r^{3+\gamma}f(t,x_0+rx+tv_0, v_0+rv),\]
	A direct calculation shows that $\tilde f$ is a solution of the Boltzmann equation with $\tilde f \geq \ell r^{3+\gamma}$ for $(t,x,v) \in [0,t_0]\times B_1(0)\times B_1(0)$.
	
	Let $T_0\in (0,t_0]$ be arbitrary, and define the following sequences for $n\geq 1$:
	\[\begin{split}
	T_n &:= \left(1 - \frac 1 {2^n}\right)T_0,\quad
	\xi_n := \frac 1 {2^{n+1}},\quad
	\rho_n := \frac 1 {2^n}, \quad
	R_n := \sqrt 2 (1-\xi_n)R_{n-1}, \quad\text{and}\quad R_0 = 1.
	\end{split}\]
	Note that $R_n \approx 2^{n/2}$. Letting $\ell_0 = \ell r^{3+\gamma}$, our initial lower bound is $\tilde f\geq \ell_0$ for $t\in [T_0/2,T_0] = [T_1,T_0]$, $x\in B_{\rho_0} = B_1$, and $v\in B_{R_0} = B_1$.
	
	Applying Lemma \ref{l:spread} iteratively, we obtain a sequence of lower bounds $\ell_n>0$ such that $\tilde f \geq \ell_n$ on $[T_n,T_0]\times B_{\rho_n}\times B_{R_n}$. The smallness condition on $\xi_n$ in Lemma \ref{l:spread} holds at every stage because $\xi_n^q R_n^{3+\gamma}\ell_n \leq (2^{-n})^{q - (3+\gamma)/2} < 1/2$. (If necessary, we can replace $\ell_n$ with $\min\{\ell_n,1\}$.)  Notice that% $\ell_{n+1}$ satisfies
	\[ \begin{split}
	\ell_{n+1} &= c \xi_n^q R_n^{3+\gamma} \ell_n^2 \min\{T_{n+1}-T_n, R_n^{-(\gamma+2s)_+ + 2s}\xi_n^{2s} + R_n^{-1}\rho_n\}\\
	&=  c \xi_n^q R_n^{3+\gamma} \ell_n^2 \min\{2^{-n-1}t_0, R_n^{-(\gamma+2s)_++2s} 2^{-2s(n+1)} + R_n^{-1} 2^{-n}  \}
	\geq c 2^{-Cn} T_0 \ell_n^2, 
	\end{split}\]
	for some constants $c,C>0$. Iterating this inequality, we obtain $\ell_n \geq u^{2^n}$ for some $u\in (0,1)$. Since the lower bound $\ell_n$ holds for $|v| \leq C 2^{n/2}$ and $|x|\leq 2^{-n}$, the Gaussian decay of $\tilde f(T_0,0,v)$ follows.
	
	Since $T_0\in (0,t_0]$ was arbitrary, we translate from $\tilde f$ back to $f$ and obtain $f(t,x_0+tv_0,v) \geq c_1 e^{-c_2|v-v_0|^2}$ for all $t\in (0,t_0]$, with $c_1$ and $c_2$ as in the statement of the proposition.
	
	 Applying the above argument with $r/2$ replacing $r$ and arbitrary $x\in B_{r/2}(x_0)$ replacing $x_0$, we conclude $f(t,x+tv_0,v)\geq c_1 e^{-c_2|v-v_0|^2}$ for $t\in (0,t_0]$, and the proof is complete.
\end{proof}
\subsection{Proof of the main theorem}\label{s:vacuum}

\begin{proof}[Proof of Theorem \ref{t:main}] 
	The proof proceeds in five steps.
	
	\textbf{Step 1: sustaining mass for a small time.} By assumption, $f_{\rm in}(x,v) \geq \delta 1_{B_r(x_0)\times B_r(v_0)}$. Lemma \ref{l:push} with $\tau = 1$ implies 
	\[ f(t,x,v) \geq \frac \delta 2, \quad \mbox{ if } |v-v_0|<r/4, |x-x_0-tv_0|<r/4, 0\leq t\leq t_*,\]
	where 
	\[ t_* = \min\left\{\frac 1 2, \frac {C_{\ref{l:push}} r^{2s}}{\langle |v_0|+r\rangle^{(\gamma+2s)_+}},  T \right\}. \]
	In particular, $t\leq 1/2$ implies $|x-x_0-tv|<3r/8$ if $|x-x_0-tv_0|< r/4$.

	{\bf Step 2: spreading mass to all $v$ (localized in $x$) for small times.}  Applying Proposition \ref{p:gaussian} with $\ell = \delta/2$ and  $r/2$ replacing $r$, we obtain
	\begin{equation}\label{e:gaussian}
	 f(t,x,v) \geq c_1\delta e^{-c_2|v-v_0|^2}, \quad 0\leq t\leq t_*, |x-x_0-tv_0|<r/8,
	 \end{equation}
	with $c_1, c_2$ depending on $r$.
	
	{\bf Step 3: spreading mass in $x$ for small times.} Now, fix any $x_1 \in \R^3$ and $t_1$ satisfying 
	\[0<t_1\leq \min\left\{t_*, \frac r {16|v_0|}\right\}.\]
	 The triangle inequality implies that at time $t_1/2$, the lower bound \eqref{e:gaussian} holds for $|x-x_0|< r/16$. Let $v_1 = 2(x_1-x_0)/t_1$, and let $\delta_0>0$ be such that %depending on $r_0$ (as well as $r$, $v_0$, $x_1$, $t_1$, and $\delta$) such that 
	\[ f(t_1/2, x, v) \geq \delta_0, \quad \mbox{ if } |x-x_0|< r_0, |v-v_1|< r_0/t_1,\] 
	where $r_0 = r/16$. We aim to use \Cref{l:push} with $v_0 = v_1$ and $\tau = t_1/2$, applied to $f(t_1/2+t, x,v)$, to propagate this lower bound along trajectories $x \sim x_0 + t v_1$, up to $t=t_1/2$. Therefore, we require
		\begin{equation}\label{e:t1}
		  	%\frac{t_1}{2} < \frac{C r^{2s}}{\tau^{2s} \langle |v_1| + r/\tau\rangle^{(\gamma+2s)_+}}
		  	%\quad\text{or, equivalently,}\quad
		  	t_1^{1+2s}  \left\langle \frac{|x_1-x_0| +r_0}{t_1}\right\rangle^{(\gamma+2s)_+} < C_{\ref{l:push}} r_0^{2s} \qquad \text{(the last condition in \Cref{l:push})}.
		 \end{equation}
If $t_1$ satisfies this inequality, then Lemma \ref{l:push} implies
\[ f(t, x,v) \geq \frac {\delta_0} 2, \quad  \mbox{ if }\frac{t_1}2 <t<t_1, |v-v_1|<\frac {r_0} 4, |x-x_0-t v_1|<\frac {r_0} {4}.\]
 Applying Proposition \ref{p:gaussian} to $f(t_1/2+t,x,v)$, we conclude $f(t_1,x_1,v) \geq c_1 e^{-c_2|v|^2}$, for some $c_1,c_2>0$ depending on $r$, $|x_1-x_0|$, $|v_0|$, and $\delta$.

{\bf Step 4: Extending the lower bound for moderate times.}  On the other hand, if $t_1$ does not satisfy \eqref{e:t1}, choose $\tilde t_1$ sufficiently small depending on $r$ and $|x_1-x_0|$ so that the inequality is satisfied. (This is always possible, since $\gamma<1$.) Proceeding as above, with $\tilde t_1$ replacing $t_1$, we obtain a lower bound at $t=\tilde t_1$, $x$ near $x_1$, and (via Proposition \ref{p:gaussian}) velocities near zero:
 \[ f(\tilde t_1, x, v) \geq \delta_1, \quad \mbox{ if } |x-x_1|<\frac {r_0} 4, |v|< \frac {r_0} 4.\]
 Next, we propagate this forward in time with $v_0 = 0$ and $\tau = 1$ in Lemma \ref{l:push} to obtain
	\[
		f(t,x,v) \geq \frac {\delta_1} 2,
			\qquad \mbox{ if } |x-x_1|< \frac {r_0} {16},
			\ |v|<\frac {r_0} {16},
			\text{ and }
			 t\leq \min\{\tilde t_1 + T_*, T\},
	\]
	where, with $C$ from the last condition in \Cref{l:push},
	\[
		T_*
			:= C_{\ref{l:push}} (r_0/4)^{2s}\langle r_0/4\rangle^{-(\gamma+2s)_+}.
	\]
	As long as $t_1\leq \min\{\tilde t_1 + T_*,T\}$, this lower bound extends up to time $t_1$, and applying Proposition \ref{p:gaussian} to $f(\tilde t_1 + t,x,v)$, we obtain $f(t_1,x_1,v) \geq ce^{-c|v-v_1|^2} \geq C^{-1} c e^{-c|v|^2}$, with $c$ depending on $\delta, r,v_0,t_1$, and $|x_1-x_0|$. In particular, if the initial data is well-distributed, then $c$ can be chosen depending only on $t_1$ and the well-distributed parameters $\delta$, $r$, and $R$.
	
	{\bf Step 5: extending the lower bound for any $t$.}  We have established a Gaussian lower bound on $f(t_1,x_1,v)$, where $x_1\in \R^3$ is arbitrary and $t_1 \leq T_0: =\min\{t_*, r/(16|v_0|), T_*,T\}$. The upper bound for $t_1$ depends only on $|v_0|$ and $r$ (not on $x_1$). Therefore, we have shown	  
%	  Since $t_1\in (0,t_*]$ and $x_1\in \R^3$ were arbitrary, we conclude
	\begin{equation}\label{e:c11}
		f(t,x,v) \geq \mu(t,x) e^{-\eta(t,x)|v|^2}, \quad 0\leq t\leq T_0, x\in \R^3, v\in \R^3.
	\end{equation}
	It is clear from our construction that $\mu$ and $\eta$ are uniformly positive on compact subsets of $(0,T_0]\times \R^3$ and that $T_0$ depends only on $r$ and $v_0$.  Since $f(T_0/2,\cdot,\cdot)$ satisfies~\eqref{e:mass-core} with the same $r$ and $v_0$ (although a different $\delta$), we can repeat steps one through four to obtain~\eqref{e:c11} on $[T_0/2,3T_0/2]$.  Iterating this finitely many times concludes the proof of the first statement in the theorem. %  In particular, for any $L>0$, we have $f(T_0,x,v) \geq \mu_L e^{-\eta_L |v|^2}$ for all $|x|<L$, for some uniform constants $\mu_L, \eta_L>0$.  Applying \Cref{l:push} to $f(T_0+t,x,v)$ with $r=L$, $x_0 = v_0 = 0$, and $\tau = \sqrt L$, {\CH we obtain a uniform lower bound for $f(t,x,v)$ with $|v|<\sqrt L$, $|x|<L$, and $t\leq T_0 + T_L$, for some $T_L$ that can be made arbitrarily large by increasing $L$ (cf.~the last condition in \Cref{l:push}) and then we may reapply Step 2.  This concludes the proof of the first statement in the theorem.
	
	The second statement follows by noticing that the lower bound in Step 3 is uniform in $x$ due to the well-distributed condition, and, hence, all subsequent bounds are independent of $x$. 
	 %depending only on the constant $C$ in Lemma \ref{l:push}. Taking $L$ sufficiently large, the lower bounds can reach the neighborhood of any $(t,x)\in [T_0,T_0+T_1]\times \R^3$ (at the cost of decreasing $\mu_L$ and $\nu_L$). Repeating finitely many times, we obtain a similar lower bound up to $t=T$.  Finally, Proposition \ref{p:gaussian} applied at every $(t,x)\in [T_0,T]\times \R^3$ implies the first statement of the theorem.
	%
%	{\CH CAN MAYBE REMOVE THIS PARAGRAPH.} In the case that $f_{\rm in}$ is well-distributed, the lower bounds at $t=T_0$ are uniform in $x$, so for any $x_1\in \R^3$, we may apply the reasoning of the previous paragraph with $|x-x_1|<1$ replacing $|x|<L$ to establish lower bounds that are uniform in $x$ up to time $t\in[T_0,T]$. {\CH COME BACK TO THIS}
\end{proof}

\section{Upper bounds and continuation}\label{s:continuation}

In this section, we apply the lower bounds of \Cref{t:main} to derive an improved continuation criterion. 
%
%\subsection{Coercivity of $Q_{\rm s}$}\label{s:coercive}
%
First, we show that our pointwise lower bounds for $f$ imply coercivity estimates for the linear operator $Q_{\rm s}(f,\cdot)$. The following lemma plays a similar role to \cite[Lemma A.3]{imbert2016weak} or \cite[Lemma 4.8]{silvestre2016boltzmann}, but the key difference is that here, no bound on the entropy density is used. %\sout{(See also the recent works of Gamba-Alonso on cut-off Boltzmann, where a coercive lower bound for the loss term is proven using pointwise lower bounds for $f$---this also replaced earlier estimates that used entropy.} %{\color{blue} [Need to track down these citations. Also, maybe this parenthetical should go in the introduction.]} {\CH I think it is better here since it is so specific.}) 
%{\color{blue}[Actually, I looked this up and their coercivity estimates aren't that similar to ours--they use upper and lower bounds on the mass and energy, not pointwise lower bounds for $f$--so I think we can forget about citing them.]}

\begin{lemma}\label{l:coercive}
	Let $f:\R^3\to \R$ be a nonnegative function with $f(v) \geq \delta 1_{B_r(v_0)}$ for some $\delta, r > 0$ and $v_0\in \R^3$. Then there exist constants $\lambda, \mu > 0$ (depending on $\delta$, $r$, and $|v_0|$) such that for all $v\in \R^3$, there is a symmetric subset of the unit sphere $A(v)\subset \mathbb S^2$ such that
	\begin{itemize}
		\item $|A(v)|>\mu(1+|v|)^{-1}$, where $|\cdot|$ is the $2$-dimensional Hausdorff measure.
		\item 	$K_f(v,v') \geq \lambda (1+|v|)^{1+\gamma+2s} |v'-v|^{-3-2s}$ whenever $(v-v')/|v-v'| \in A(v)$.
	\end{itemize}
\end{lemma}

\begin{proof}
	From \eqref{e:kernel}, letting $w = v_*' - v$, we have
	\[ K_f(v,v') \geq \frac{\delta}{C} \left(\int_{\{w\cdot (v'-v) = 0\}} 1_{B_r(v_0)}(v+w) |w|^{\gamma+2s+1} \dd w\right) |v-v'|^{-3-2s}.\]
	For fixed $v\in \R^3$, the integral in parentheses is only nonzero if the plane $P(v') = v + (v-v')^\perp$ intersects $B_r(v_0)$ nontrivially. To get a uniform positive lower bound, we take $A(v)$ to be the set of directions $(v-v')/|v-v'|$ such that $P(v')$ intersects $B_{r/2}(v_0)$. If $(v-v')/|v-v'| \in A(v)$, then 
	\[
	\begin{split}
	\int_{\{w\cdot (v'-v) = 0\}} 1_{B_r(v_0)}(v+w) |w|^{\gamma+2s+1} \dd w &\geq C^{-1}|v_0-v|^{\gamma+2s+1} |B_r(v_0)\cap \{w\cdot (v'-v) = 0\}| \\
	&\geq C^{-1} r^2 |v_0-v|^{\gamma+2s+1}.
	\end{split}
	\]
	To estimate the size of $A(v)\subset \mathbb S^2$, note that $\omega\in A(v)$ if and only if there is a vector $\beta \perp \omega$ with the angle $\varphi$ between $\beta$ and $v_0-v$ satisfying $|\sin \varphi| \leq (r/2)/|v-v_0|$. This means $\omega$ lives in a strip of width $\approx r |v-v_0|^{-1}$ centered around the equator perpendicular to $v_0-v$ (i.e. the intersection of $(v_0-v)^\perp$ with $\mathbb S^2$). This strip has surface area $\geq C^{-1} r|v-v_0|^{-1}$.
\end{proof}

Now we may follow the proof of \cite[Theorem 1.2]{silvestre2016boltzmann}, with our Lemma \ref{l:coercive} replacing \cite[Lemma 4.8]{silvestre2016boltzmann}, to obtain: 
\begin{proposition}
	Let $f$ be a solution of the Boltzmann equation \eqref{e:Boltzmann} on $[0,T]\times \mathbb T^3\times \R^3$ that satisfies \eqref{e:hydro}. Assume that the initial data $f_{\rm in}$ satisfies~\eqref{e:mass-core} for some $\delta, r>0$, $x_0 \in \mathbb{T}^3$, and $v_0 \in \R^3$. %satisfies $f(0,x,v) \geq \delta 1_{B_r(x_0)\times B_r(v_0)}$ for some $\delta, r>0$, $x\in \mathbb T^3$, $v\in \R^3$. 
	Then $f$ satisfies an $L^\infty$ bound that is uniform away from $t=0$, i.e.
	\[
		\|f(t,\cdot,\cdot)\|_{L^\infty(\mathbb T^3\times \R^3)}
			\leq C_{T, \delta, r,v_0}(1+t^{-3/(2s)}).
	\]
	%{\CH The constant $C$ depends additionally on $T$ and, if $\gamma +2s < 0$, on 
	%\sout{The constant $K$ depends only on $T$, $M_0$, $E_0$, the initial data, and, if $\gamma + 2s < 0$, on}
	If $\gamma + 2s < 0$, $C_{T,\delta,r,v_0}$ depends additionally on
	\[ \sup_{[0,T]\times \mathbb T^3} \int_{\R^3} \vv^q f(t,x,v)^p \dd v,\]
	where $p>3/(3+\gamma+2s)$ and $q = \max\{0, 1 - 3(\gamma+2s)/(2s)\}$.
\end{proposition}

%Unlike \cite{silvestre2016boltzmann}, this bound does not require a lower bound on the mass density or an upper bound on the 
%Note that the only place that \cite{silvestre2016boltzmann} uses the entropy bound is in obtaining a cone of nondegeneracy as in our Lemma \ref{l:coercive}.
Note that \cite{silvestre2016boltzmann} uses the entropy bound to obtain a cone of nondegeneracy, but does not use the entropy bound anywhere else. Our Lemma \ref{l:coercive} produces a cone of nondegeneracy without using an entropy bound, which allows us to weaken the assumptions.

Combining this $L^\infty$ bound with the basic inequality $s\log s \leq s^2$, we see that for $t\geq t_0$,
\[ H_f(t,x) := \int_{\R^3} f(t,x,v)\log f(t,x,v)\dd v \leq \|f(t,x,\cdot)\|_{L^\infty_v(\R^3)} \int_{\R^3} f(t,x,v)\dd v \leq C,\]
for some $C$ depending only on $t_0$, $K_0$, and the initial data. (Recall that, for $\gamma +2s\in [0,2]$, $K_0$ is an upper bound for the mass and energy densities.)

We have shown that, away from $t=0$, the entropy density $H_f(t.x)$ is bounded from above and the mass density is bounded from below, in terms of only $K_0$ and the initial data. Combining this with \cite[Theorem 1.2]{imbert2019smooth} immediately implies part (a) of Corollary \ref{c:continuation}. Part (b) follows from combining these smoothing estimates with the local existence result of \cite{HST2019boltzmann} (when $\gamma < 0$) or \cite{morimoto2015polynomial} (when $\gamma \leq 0$ and $s\in (0,1/2)$). % Clearly, the above lower bounds also imply $M_f(t,x)$ is bounded below depending only on $M_0$, $E_0$, and the initial data. Then, we can remove the quantities $m_0$ and $H_0$ from the continuation criterion and regularity/decay estimates of \cite{imbert2019smooth}.

\bibliographystyle{abbrv}
\bibliography{boltzmann}
\end{document}